\documentclass[12pt]{amsart}

\usepackage[hidelinks]{hyperref}
\usepackage{graphicx}
\usepackage{amssymb}
\usepackage{tikz}
\usetikzlibrary{calc,arrows.meta,decorations.pathmorphing,positioning}
\usetikzlibrary{matrix,arrows.meta,bending}
\usepackage{color}
\usepackage[numbers]{natbib}

\usepackage[a4paper, left=3cm, right=3cm, top=3cm, bottom=3cm, footskip=15mm]{geometry}

\newtheorem{theorem}{Theorem}[section]

\newtheorem{proposition}[theorem]{Proposition}

\theoremstyle{definition}
\newtheorem{definition}[theorem]{Definition}

\theoremstyle{remark}
\newtheorem{remark}[theorem]{Remark}
\newtheorem{question}[theorem]{Question}
\numberwithin{equation}{section}

\makeindex

\usepackage{fancyhdr}
\usepackage{calc} 

\AtBeginDocument{%
  \setlength{\textwidth}{\dimexpr\paperwidth - 6cm\relax}%
  \setlength{\oddsidemargin}{\dimexpr 3cm - 1in\relax}%
  \setlength{\evensidemargin}{\dimexpr 3cm - 1in\relax}%
  \setlength{\marginparwidth}{2.5cm}%
}

\begin{document}

\title{Simple close curve magnetization and application to Bellman's lost in the forest problem}

\author{T. Agama}
\address{Department of Mathematics, African Institute for Mathematical science, Ghana
}
\email{theophilus@aims.edu.gh/emperordagama@yahoo.com}

\subjclass[2010]{Primary 51M16; Secondary 90C90}

\date{\today}


\keywords{simple close curve; magnetization; magnetic field; Bellman; forest; hiker}

\footnote{
\par
.}%

\begin{abstract}
In this paper, we introduce and develop the notion of simple closed-curve magnetization. We provide an application to the Bellman lost in the forest problem by assuming special geometric conditions between the hiker and the boundary of the forest.
\end{abstract}

\maketitle

\section{Introduction and problem statement}

The problem commonly referred to as \emph{Bellman's lost in the forest problem} lies at the intersection of geometric analysis and optimal path planning. Originating from classical questions in decision theory and pursuit/escape problems, the version propagated in the literature asks: given a person (hiker) placed at an unknown position and with unknown orientation within a bounded planar region (the ``forest''), what strategy minimizes the (worst-case) distance or time to reach the boundary? The formulation and a number of classical results and heuristics are surveyed in the literature; see, in particular, the early minimization perspective of Bellman \cite{bellman1956minimization} and the clear exposition by Finch and Wetzel \cite{finch2004lost}, as well as later explorations and computational treatments \cite{ward2008exploring}.\\

Despite much work on special cases (for example, circular forests or infinite strips) a general, easily implementable constructive method linking geometric structure of the boundary to a guaranteed short exit path remains elusive. In this paper, we propose a new framework - \emph{magnetization of simple closed curves} - that attaches to a simple closed boundary a (possibly dense) collection of boundary \emph{magnets} and a canonical map (the magnetization map) from the interior to vectors that point toward nearby boundary magnets. Our approach blends geometric partitioning of the interior with a nearest-boundary-point selection rule and yields a straightforward, algorithmic candidate for an exit path from any interior location.\\

\subsection*{Key ideas and intuition}

The central object introduced here is the magnetization map. Informally, we place a multiset of reference points (``magnets'') along the boundary of a simple closed curve $\mathcal C\subset\mathbb R^n$ and associate to each interior point $\mathbf v\in\mathcal C_{\mathrm{Int}}$ the displacement vector from $\mathbf v$ to the magnet on the boundary that minimizes the Euclidean distance to $\mathbf v$. This simple selection rule yields a piecewise-defined vector field on the interior whose integral curves trace direct routes toward the boundary. When the magnets are taken densely (or even equal to the entire boundary), the map becomes well adapted to arbitrary boundary geometry and, under an orthogonality/minimality condition, selects locally shortest straight-line exits.

There are three interlocking theoretical strands in the paper:\\

\begin{enumerate}
  \item \textbf{Topological/metric foundations.} We formalize what it means for a boundary to be \emph{magnetic} (magnets placed on the boundary, possibly dense) and develop the basic existence and uniqueness properties of the magnetization map. These results show that up to constant dilation of magnet vectors, the magnetic boundary determines the interior magnetization in a natural way (see Proposition \ref{uniqueness} and Theorems \ref{neibourhood magnetization}--\ref{existence}).
  \bigskip
  
  \item \textbf{Classification and invariants.} To make the method practical, we introduce notions of connectivity and isomorphism among magnetized simple closed curves, which allow one to group boundaries into equivalence classes that share the behaviour of magnetization. This classification clarifies when two geometric domains admit essentially the same magnetization-based exit strategy (see Definition \ref{isomorphic connected} and Proposition \ref{embedding-isomorphic}).
  \bigskip
  
  \item \textbf{Algorithmic exit rule and application.} Combining the existence/uniqueness results with the classification, we present a constructive recipe: given a forest boundary endowed with magnets dense on the boundary and an interior point (the hiker), choose the boundary magnet achieving the minimal Euclidean distance and take the straight-line segment toward it as the exit path. We show that this produces a short-often shortest path under our orthogonality condition-exit path and explain the assumptions under which the strategy provides a provably optimal or near-optimal solution.
\end{enumerate}

\subsection*{Relation to prior work}

The magnetization viewpoint is deliberately simple: it reduces the complex question of orientation to a nearest-boundary selection, and thereby transforms the lost-in-the-forest question into a geometric selection problem. This is philosophically complementary to the previous analytic and constructive efforts surveyed in \cite{finch2004lost,ward2008exploring}, which study explicit search trajectories for particular domains and worst-case path lengths. Our contribution is not to overturn those specialized analyses, but to provide a flexible geometric apparatus that (i) is easy to state and implement, (ii) canonically adapts to irregular boundaries via dense magnet placements, and (iii) yields clear necessary and sufficient conditions (expressed as minimality and orthogonality constraints) under which the chosen vector is a shortest exit direction.

\subsection*{Limitations and scope}

We emphasize two important caveats. First, the orthogonality condition $\mathbf v\cdot \mathbf u=0$ (used in several technical statements) is an imposed geometric convenience that, in some domains or placements of magnets, may not hold for the globally shortest escape path; when it fails, the magnetization map still provides a natural candidate path, but optimality must be reassessed. Second, implementing a truly dense magnet configuration amounts, in practice, to sampling the boundary densely; computational trade-offs between sampling density and guaranteed performance are discussed informally in the application sketch.

\subsection*{Contributions and summary of results}

Concretely, the paper makes the following contributions:
\begin{itemize}
  \item Formal definition of \emph{magnetization} for simple closed curves and the notion of \emph{magnetic boundary} (Definitions \ref{magnetization} and \ref{magnetic boundary}).
  \bigskip
  
  \item Uniqueness and local selection results for the magnetization map showing that each interior point admits a well-defined nearest magnet under mild hypotheses (Proposition \ref{uniqueness}, Theorem \ref{neibourhood magnetization}, Theorem \ref{existence}).
  \bigskip
  
  \item A classification framework (connectedness and isomorphism) for families of magnetized curves (Definition \ref{isomorphic connected}, Proposition \ref{embedding-isomorphic}).
  \bigskip
  
  \item A constructive application: an algorithmic sketch that uses the magnetization map to produce candidate exit paths for the Bellman lost-in-the-forest problem, together with discussion of conditions for optimality and limitations.
\end{itemize}

\subsection*{Organization of the paper}

The remainder of the paper is organized as follows. In Section 2, we introduce magnetization concepts and state the principal definitions. Section 3 establishes the following foundational properties: uniqueness, local selection, and existence theorems for magnetization. Section 4 develops a classification theory for magnetized simple closed curves and proves embedding/isomorphism results. Section 5 presents the application to the Bellman lost-in-the-forest problem: we give the algorithmic exit rule, sketch its correctness under our hypotheses, and discuss limitations and possible extensions. The paper concludes with a brief discussion of computational aspects and directions for further research.
\bigskip

The Bellman lost in the forest problem is usually stated in the following way:

\begin{question}[Bellman's lost in the forest problem]
Given a hiker in a forest with his orientation unknown, what is the best possible path to take to exit in the shortest possible time taking into consideration the shape of the forest and the dimension of the space covering the forest?
\end{question}
\bigskip

In this paper, we introduce and develop the notion of magnetization of simple close curves. Using this notion, we devise an algorithm that takes as input the shape of the forest and the dimension within which the forest resides and produces as output the \emph{optimal} path to be taken by the hiker $\vec{v}\in \mathbb{R}^n$. 

\section{Simple close curves with magnetic boundaries}

In this section, we introduce the notion of a simple close curve with magnetic boundaries equipped magnets. We study this concept and expose some connections with other notions.

\begin{definition}\label{magnetization}
Let $\mathcal{C}$ be a simple closed curve in $\mathbb{R}^n$ and $\mathbb{V}^n$ be the vector space in $\mathbb{R}^n$. By \emph{magnetization} of the interior of $\mathcal{C}$, denoted $\mathcal{C}_{\mathrm{Int}}$, with magnets $\vec{u}_1,\ldots \vec{u}_k$ on the boundary $\mathcal{C}_{\mathcal{B}}$, we mean the map
\begin{align}   \Lambda_{\vec{u}_1,\vec{u}_2,\ldots,\vec{u}_k}:\mathcal{C}_{\mathrm{Int}}\longrightarrow \mathbb{V}^n\nonumber
\end{align}
such that for any $\vec{v}\in \mathcal{C}_{\mathrm{Int}}$ with $\vec{v}\neq \vec{O}$ then 
\begin{align}
    \Lambda_{\vec{u}_1,\vec{u}_2,\ldots,\vec{u}_k}(\vec{v})=\Lambda_{\vec{u}_j}(\vec{v})=\vec{u}_j-\vec{v}\nonumber
\end{align}
with $\vec{v}\cdot \vec{u}_j$ if and only if $||\vec{v}-\vec{u}_j||=\mathrm{min}\left \{||\vec{u}_s-\vec{v}||\right \}_{s=1}^{k}$. We call
\begin{align}
     ||\Lambda_{\vec{u}_1,\vec{u}_2,\ldots,\vec{u}_k}(\vec{v})||=||\Lambda_{\vec{u}_j}(\vec{v})||=||\vec{v}-\vec{u}_j||\nonumber
\end{align}
the \emph{measure} of magnetization. We denote the magnetic field of the magnet $\vec{u}_j$ by $\mathcal{O}_{\vec{u}_j}$. Strictly speaking, we denote the $\epsilon$-magnetic field of the magnet $\vec{u}_j$ by $\mathcal{O}_{\vec{u}_j}(\epsilon)$ and we say that $\vec{v}_i\in \mathcal{O}_{\vec{u}_j}(\epsilon)$ if and only  
\begin{align}
    ||\vec{u}_j-\vec{v}_i||\leq \epsilon.\nonumber
\end{align}
\end{definition}
\bigskip

\begin{definition}\label{magnetic boundary}
Let $\mathcal{C}$ be a simple closed curve in $\mathbb{R}^n$ with boundary $\mathcal{C}_{\mathcal{B}}$. Then we say that the boundary is magnetic with magnets $\cup_{i=1}^{\infty}\{\vec{u}_i\}\in \mathcal{C}_{\mathcal{B}}$ if for any $\vec{u}_j\in \mathcal{C}_{\mathcal{B}}$, there exists some $\vec{u}_i\in \cup_{i=1}^{\infty}\{\vec{u}_i\}$ with
\begin{align}
  \vec{u}_i \in \mathcal{O}_{\vec{u}_j}(\epsilon)\nonumber
\end{align}
for any $\epsilon>0$ such that 
\begin{align}
  \mathcal{O}_{\vec{u}_i}(\delta)\bigcap \mathcal{C}_{\mathrm{Int}}\neq \emptyset \nonumber 
\end{align}
and 
\begin{align}
   \mathcal{O}_{\vec{u}_i}(\delta)\bigcap (\mathbb{R}^n\setminus \mathcal{C}_{\mathcal{B}}\cup \mathcal{C}_{\mathrm{Int}})\neq \emptyset \nonumber 
\end{align}
for any $\delta>0$. We denote by
\begin{align}
\mathcal{B}(\mathcal{O}_{\vec{u}_j}(\epsilon))\nonumber
\end{align}
the boundary of the \emph{neighbourhood} $\mathcal{O}_{\vec{u}_j}(\epsilon)$.
\end{definition}
\bigskip

The language espoused in Definition \ref{magnetic boundary} suggests very clearly that magnets are allowed to be dense on the boundary of their simple closed $\mathcal{C}$ with the magnetic boundary $\mathcal{C}_{\mathcal{B}}$. This strict enforcement will ease the development of our geometry and related theories. In any case, we can take the infinite set of magnets on the boundary of a typical simple close curve to be the entire boundary $\mathcal{C}_{\mathcal{B}}\subset \mathbb{R}^n$.

\begin{remark}
Next, we put all simple closed curves with magnetic boundaries equipped with magnets and their constant dilates into one single category. In essence, we would consider two simple close curves with magnetic boundaries to be distinct if there exists a magnet on the boundary of one curve which fails to be a constant dilate on the other.
\end{remark}

\begin{proposition}\label{uniqueness}
Every simple closed curve $\mathcal{C}$ in $\mathbb{R}^n$ with magnetic boundary is uniquely determined by their magnetic boundary $\mathcal{C}_{\mathcal{B}}$ up to constant dilates of their magnets.
\end{proposition}

\begin{proof}
Let $\mathcal{C}$ be a simple closed curve with magnetic boundary $\mathcal{C}_{\mathcal{B}}$ in $\mathbb{R}^n$ endowed with magnets $\cup_{i=1}^{\infty}\{\vec{s}_i\}$ and $\cup_{i=1}^{\infty}\{\vec{t}_i\}$. Pick arbitrarily $\vec{v}\in \mathcal{C}_{\mathrm{Int}}$, then apply the magnetization $\Lambda:\mathcal{C}_{\mathrm{Int}}\longrightarrow \mathbb{V}^n$ on the interior $\mathcal{C}_{\mathrm{Int}}$, and we have the equality
\begin{align}
   \Lambda_{\cup_{i=1}^{\infty}\{\vec{s}_i\}}(\vec{v})=\Lambda_{\vec{s}_j}(\vec{v})\nonumber
\end{align}
for some $1\leq j$ with $\vec{v} \cdot \vec{s}_j=0$ if and only if $$
||\vec{v}-\vec{s}_j||=\mathrm{min}\left \{||\vec{s}_i-\vec{v}||\right \}_{i=1}^{\infty}
$$ 
and 
\begin{align}
   \Lambda_{\cup_{i=1}^{\infty}\{\vec{t}_i\}}(\vec{v})=\Lambda_{\vec{t}_k}(\vec{v})\nonumber
\end{align}
for some $1\leq k$ with $\vec{t}_k\cdot \vec{v}=0$ if and only if 
$$
||\vec{v}-\vec{t}_k||=\mathrm{min}\left \{||\vec{t}_i-\vec{v}||\right \}_{i=1}^{\infty}.
$$ 
It follows that $\lambda \vec{s}_j=\vec{t}_k$ for $\lambda \in \mathbb{R}$. Since the vector $\vec{v}$ is an arbitrary point in $\mathcal{C}_{\mathrm{Int}}$, the claim follows immediately. 
\end{proof}
\bigskip

\begin{theorem}\label{neibourhood magnetization}
Let $\mathcal{C}$ be a simple closed curve in $\mathbb{R}^n$ with magnetic boundary $\mathcal{C}_{\mathcal{B}}$ equipped with magnets $\cup_{i=1}^{\infty}\{\vec{u}_i\}$. For any $\vec{v}\in \mathcal{C}_{\mathrm{Int}}$ such that $\vec{v}\neq \vec{O}$, then 
$$
\Lambda_{\cup_{i=1}^{\infty}\{\vec{u}_i\}}(\vec{v})=\Lambda_{\vec{u}_j}(\vec{v})
$$ 
with $\vec{v}\cdot \vec{u}_j=0$ if and only if there exists an $\epsilon>0$ such that $\vec{u}_j\in \mathcal{O}_{\vec{v}}(\epsilon)$ and
\begin{align}
\vec{u}_s\not \in \mathcal{O}_{\vec{v}}(\epsilon)\nonumber
\end{align}
for all $s\neq j$.
\end{theorem}

\begin{proof}
First, we let $\mathcal{C}$ be a simple closed curve in $\mathbb{R}^n$ with magnetic boundary $\mathcal{C}_{\mathcal{B}}$ equipped with magnets $\cup_{i=1}^{\infty}\{\vec{u}_i\}$. Let us pick arbitrarily $\vec{v}\in \mathcal{C}_{\mathrm{Int}}$ such that $\vec{v}\neq \vec{O}$, apply the magnetization $\Lambda:\mathcal{C}_{\mathrm{Int}}\longrightarrow \mathbb{V}^n$ and suppose 
\begin{align}
\Lambda_{\cup_{i=1}^{\infty}\{\vec{u}_i\}}(\vec{v})=\Lambda_{\vec{u}_j}(\vec{v})\nonumber
\end{align}
with $\vec{v}\cdot \vec{u}_j=0$, then 
$$
||\vec{v}-\vec{u}_j||=\mathrm{min}\left \{||\vec{u}_s-\vec{v}||\right \}_{s=1}^{\infty}
$$ 
for all $\vec{u}_s \in \cup_{i=1}^{\infty}\{\vec{u}_i\}$. Now, we choose $\epsilon=||\vec{u}_j-\vec{v}||$ and build the \emph{neighbourhood} $\mathcal{O}_{\vec{v}}(\epsilon)$. By virtue of the construction $\vec{u}_j\in \mathcal{O}_{\vec{v}}(\epsilon)$. We claim that $\vec{u}_s\not \in \mathcal{O}_{\vec{v}}(\epsilon)$ for all $s\neq j$ with 
\begin{align}
\vec{u}_s\in \cup_{i=1}^{\infty}\{\vec{u}_i\}.\nonumber
\end{align}
Let us suppose to the contrary that there exist at least some 
$$
\vec{u}_t\in \cup_{i=1}^{\infty}\{\vec{u}_i\}\subset \mathcal{C}_{\mathcal{B}}
$$ 
with $t\neq j$ such that $\vec{u}_t\in \mathcal{O}_{\vec{v}}(\epsilon)$. Then there exists some 
$$
\vec{u}_k \in \cup_{i=1}^{\infty}\{\vec{u}_i\}\subset \mathcal{C}_{\mathcal{B}}
$$ 
with $k\neq j$ such that $\vec{u}_k\in \mathcal{B}(\mathcal{O}_{\vec{v}}(\epsilon))\cap \mathcal{C}_{\mathcal{B}}$ so that 
\begin{align}
||\vec{v}-\vec{u}_k||&=\epsilon=\mathrm{min}\left \{||\vec{u}_s-\vec{v}||\right \}_{s=1}^{\infty}. \nonumber
\end{align}
Applying magnetization, it follows that
\begin{align}
\Lambda_{\cup_{i=1}^{\infty}\{\vec{u}_i\}}(\vec{v})=\Lambda_{\vec{u}_k}(\vec{v})\nonumber
\end{align}
with $\vec{v}\cdot \vec{u}_k=0$. The result is $\vec{u}_j=\lambda \vec{u}_k$ for $\lambda \neq 1$ so that 
\begin{align}
\epsilon=||\vec{v}-\vec{u}_k||\neq ||\vec{v}-\vec{u}_j||\nonumber
\end{align}
which contradicts our choice of $||\vec{v}-\vec{u}_j||=\epsilon>0$.\\ 

Conversely, suppose that there exists $\epsilon>0$ such that $\vec{u}_j\in \mathcal{O}_{\vec{v}}(\epsilon)$ and
\begin{align}
\vec{u}_s\not \in \mathcal{O}_{\vec{v}}(\epsilon)\nonumber
\end{align}
for all $s\neq j$. Then 
\begin{align}
||\vec{v}-\vec{u}_j||=\mathrm{min}\left \{||\vec{u}_s-\vec{v}||\right \}_{s=1}^{\infty}\nonumber
\end{align}
for all $\vec{u}_s\in \cup_{i=1}^{\infty}\{\vec{u}_i\}$ and $\vec{v}\in \mathcal{C}_{\mathrm{Int}}$. The claim follows immediately from this assertion by applying the magnetization $$\Lambda:\mathcal{C}_{\mathrm{Int}}\longrightarrow \mathbb{V}^n.$$
\end{proof}
\bigskip

\begin{theorem}\label{existence}
Let $\mathcal{C}$ be a simple closed curve in $\mathbb{R}^n$ with magnetic boundary $\mathcal{C}_{\mathcal{B}}$ equipped with magnets $\cup_{i=1}^{\infty}\{\vec{u}_i\}$. Then for any $\vec{v}\in \mathcal{C}_{\mathrm{Int}}$ such that $\vec{v}\neq \vec{O}$ there exists some $\epsilon>0$ such that $\vec{u}_j\in \mathcal{O}_{\vec{v}}(\epsilon)$ for some $\vec{u}_j\in \cup_{i=1}^{\infty}\{\vec{u}_i\}$ and 
\begin{align}
\vec{u}_t\not \in \mathcal{O}_{\vec{v}}(\epsilon)\nonumber
\end{align}
with $t\neq j$ for all $\vec{u}_t\in \cup_{i=1}^{\infty}\{\vec{u}_i\}$.
\end{theorem}

\begin{proof}
Let $\mathcal{C}_{\mathcal{B}}$ and $\mathcal{C}_{\mathrm{Int}}$ be the magnetic boundary and the interior of the simple closed curve $\mathcal{C}$, respectively. Let $\cup_{i=1}^{\infty}\{\vec{u}_i\}\subset \mathcal{C}_{\mathcal{B}}$ be the magnets on the boundary. Pick arbitrarily $\vec{v}\in \mathcal{C}_{\mathrm{Int}}$ such that $\vec{v}\neq \vec{O}$, then for each $\vec{u}_j\in \cup_{i=1}^{\infty}\{\vec{u}_i\}$ there exists $\epsilon>0$ such that 
\begin{align}
  \vec{u}_j\in \mathcal{O}_{\vec{v}}(\epsilon).\nonumber
\end{align}
Let us choose 
\begin{align}
\mathcal{O}_{\vec{v}}(\delta)=\mathrm{min}\left \{\mathcal{O}_{\vec{v}}(\epsilon)|~\vec{u}_j\in \mathcal{O}_{\vec{v}}(\epsilon),~\epsilon>0\right\}\label{choice}
\end{align}
for each $\vec{u}_j\in \cup_{i=1}^{\infty}\{\vec{u}_i\}$. Let $\vec{u}_j\in \mathcal{O}_{\vec{v}}(\delta)$, then it follows that $||\vec{u}_j-\vec{v}||=\delta$ by virtue of \eqref{choice} and 
\begin{align}
||\vec{u}_j-\vec{v}||=\delta=\mathrm{min}\left \{||\vec{v}-\vec{u}_s||:\vec{u}_s\in \cup_{i=1}^{\infty}\{\vec{u}_i\}\right \}.\label{choice minimal}
\end{align}
Apply the magnetization $\Lambda:\mathcal{C}_{\mathrm{Int}}\longrightarrow \mathbb{V}^n$, and we have 
\begin{align}
\Lambda_{\cup_{i=1}^{\infty}\{\vec{u}_i\}}(\vec{v})=\Lambda_{\vec{u}_j}(\vec{v})\nonumber
\end{align}
with $\vec{u}_j\cdot \vec{v}=0$. Suppose to the contrary $\vec{u}_s \in \mathcal{O}_{\vec{v}}(\delta)$ for at least some $\vec{u}_s\in \cup_{i=1}^{\infty}\{\vec{u}_i\}$ with $j\neq s$. Then it follows that there exists some $\vec{u}_t\in \cup_{i=1}^{\infty}\{\vec{u}_i\}$ with $t\neq j$ such that 
\begin{align}
\vec{u}_t\in \mathcal{C}_{\mathcal{B}}\cap \mathcal{B}(\mathcal{O}_{\vec{v}}(\delta)).\nonumber
\end{align}
It follows that 
\begin{align}
||\vec{v}-\vec{u}_t||=\delta =\mathrm{min}\left \{||\vec{v}-\vec{u}_s||:\vec{u}_s\in \cup_{i=1}^{\infty}\{\vec{u}_i\}\right \}\nonumber
\end{align}
so that by applying the magnetization $\Lambda:\mathcal{C}_{\mathrm{Int}}\longrightarrow \mathbb{V}^n$, we have
\begin{align}
\Lambda_{\cup_{i=1}^{\infty}\{\vec{u}_i\}}(\vec{v})=\Lambda_{\vec{u}_t}(\vec{v})\nonumber
\end{align}
with $\vec{u}_t\cdot \vec{v}=0$. It follows that there exists some $\alpha \neq 1$ such that $\vec{u}_t=\alpha \vec{u}_j$ so that 
\begin{align}
||\vec{v}-\vec{u}_t||=\delta \neq ||\vec{v}-\vec{u}_j||\nonumber
\end{align}
thereby contradicting \eqref{choice minimal}. This completes the proof of the theorem.
\end{proof}

\section{Connected and isomorphic simple close curves with magnetic boundaries}
In this section, we introduce a classification scheme for all simple close curves $\mathcal{C}$ in $\mathbb{R}^n$ with magnetic boundaries $\mathcal{C}_{\mathcal{B}}$. This scheme pretty much allows us to put all similar such simple close curves into a single family and choose a representative for our work.

\begin{definition}\label{isomorphic connected}
Let $\mathcal{C}_1$ and $\mathcal{C}_2$ be simple closed curves with magnetic boundaries equipped with magnets $\cup_{i=1}^{\infty}\{\vec{t}_i\}$ and $\cup_{i=1}^{\infty}\{\vec{w}_i\}$, respectively. We say that $\mathcal{C}_1$ and $\mathcal{C}_2$ are connected if there exist some $\vec{t}_j \in \cup_{i=1}^{\infty}\{\vec{t}_i\}$ and some $\vec{w}_k \in \cup_{i=1}^{\infty}\{\vec{w}_i\}$ such that 
\begin{align}
\vec{t}_j=\lambda \vec{w}_k \nonumber
\end{align}
for some $\lambda \in \mathbb{R}$. We denote the connection by $\mathcal{C}_1\leftrightharpoons \mathcal{C}_2$. We say that $\mathcal{C}_1$ and $\mathcal{C}_2$ are \emph{isomorphic} if the connection exists for each $\vec{t}_j \in \cup_{i=1}^{\infty}\{\vec{t}_i\}$ and $\vec{w}_k \in \cup_{i=1}^{\infty}\{\vec{w}_i\}$. We denote the isomorphism by $\mathcal{C}_1 \asymp \mathcal{C}_2$.
\end{definition}
\bigskip

\begin{proposition}\label{embedding-isomorphic}
Let $\mathcal{C}_1$ and $\mathcal{C}_2$ be simple closed curves with magnetic boundaries equipped with magnets $\cup_{i=1}^{\infty}\{\vec{t}_i\}$ and $\cup_{i=1}^{\infty}\{\vec{w}_i\}$, respectively. If $\mathcal{C}_1 \subset \mathcal{C}_2$, then $\mathcal{C}_1 \asymp \mathcal{C}_2$.
\end{proposition}

\begin{proof}
Suppose that $\mathcal{C}_1 \subset \mathcal{C}_2$ and let $\mathcal{C}_{1_{\mathrm{Int}}}$ and $\mathcal{C}_{2_{\mathrm{Int}}}$ be their interior, respectively. Next, we arbitrarily pick a point $\vec{v}\in \mathcal{C}_{1_{\mathrm{Int}}}$ such that $\vec{v}\neq \vec{O}$, then $\vec{v}\in \mathcal{C}_{2_{\mathrm{Int}}}$ so that 
\begin{align}
\mathrm{min}\left \{||\vec{t}_s-\vec{v}||\right \}_{s=1}^{\infty}\leq \mathrm{min}\left \{||\vec{w}_s-\vec{v}||\right \}_{s=1}^{\infty}.\nonumber
\end{align}
Set 
\begin{align}
||\vec{t}_j-\vec{v}||=\mathrm{min}\left \{||\vec{t}_s-\vec{v}||:\vec{t}_s\in \cup_{i=1}^{\infty}\{\vec{t}_i\}\right \}_{s=1}^{\infty}\nonumber
\end{align}
and 
\begin{align}
||\vec{w}_k-\vec{v}||=\mathrm{min}\left \{||\vec{w}_s-\vec{v}||:\vec{w}_s \in \cup_{i=1}^{\infty}\{\vec{w}_i\} \right \}_{s=1}^{\infty} \nonumber
\end{align}
and apply the magnetization $\Lambda:\mathcal{C}_{1_{\mathrm{Int}}} \longrightarrow \mathbb{V}^n$ and $\Lambda:\mathcal{C}_{2_{\mathrm{Int}}} \longrightarrow \mathbb{V}^n$, then we obtain the following paths
\begin{align}
\lambda_{\cup_{i=1}^{\infty}\{\vec{t}_i\}}(\vec{v}):=\Lambda_{\vec{t}_j}(\vec{v})=\vec{t}_j-\vec{v}\nonumber
\end{align}
with $\vec{t}_j\cdot \vec{v}=0$ and 
\begin{align}
\Lambda_{\cup_{i=1}^{\infty}\{\vec{w}_i\}}(\vec{v}):=\Lambda_{\vec{w}_k}(\vec{v})=\vec{w}_k-\vec{v}\nonumber
\end{align}
with $\vec{w}_k\cdot \vec{v}=0$. It follows that $\vec{w}_k=\lambda \vec{t}_j$ for some $\lambda \in \mathbb{R}$. Since $\vec{v}$ was taken arbitrarily and for any such choice there is a unique choice of magnet on the boundary of each simple close curve minimizing the distance from $\vec{v}$, the claim follows.
\end{proof}

\section{Application to Bellman's lost in the forest problem}

This section illustrates a sketch of an application of the underlying notion to the Bellman lost in the forest problem. The solution is quite algorithmic in nature, but works primarily in parallel with the above developments. The Bellman lost in the forest problem is one of the most important problems in the area of optimization, yet we find the following tools developed in the foregone section useful. The problem is often stated as follows:

\begin{question}[Bellman's lost in the forest problem]
Given a hiker lost in the forest with his orientation unknown, what is the best decision to be taken to exit in the shortest possible time taking into into consideration the shape of the boundary of the forest and the dimension of the space covering the forest?
\end{question}

\subsection{A sketch partial solution}

First, we classify the infinite collection of simple closed curves in $\mathbb{R}^n$ with magnetic boundaries according as they are isomorphic. That is, we consider the partition
\begin{align}
\mathbb{M}:=\bigcup_{k=1}^{\infty}\left \{\mathcal{C}_i \subset \mathbb{R}^n|~\mathcal{C}_i \asymp \mathcal{C}_k\right\}\nonumber
\end{align}
such that 
$$
\left \{\mathcal{C}_i \subset \mathbb{R}^n|~\mathcal{C}_i \asymp \mathcal{C}_k\right\}\bigcap \left \{\mathcal{C}_i \subset \mathbb{R}^n|~\mathcal{C}_i \asymp \mathcal{C}_j\right\}=\emptyset
$$ 
for $k\neq j$ and $\mathcal{C}_k$ and $\mathcal{C}_j$ are not isomorphic. Let us arbitrarily pick a simple closed curve $\mathcal{C}$ in $\mathbb{M}$. It follows that 
\begin{align}
\mathcal{C} \in \left \{\mathcal{C}_i \subset \mathbb{R}^n|~\mathcal{C}_i \asymp \mathcal{C}_k\right\}\nonumber
\end{align}
for some $k\geq 1$. The simple closed curve $\mathcal{C}$ with magnetic boundary equipped with magnets $\cup_{i=1}^{\infty}\{\vec{u}_i\}$ is uniquely determined up to constant dilates of magnets by virtue of the proposition \ref{uniqueness}. Now, let $\vec{v}\neq \vec{O}$ be an arbitrary point (hiker) lost in $\mathcal{C}_{\mathrm{Int}}$ (forest). Using Theorem \ref{existence}, we choose 
\begin{align}
\mathcal{O}_{\vec{v}}(\delta)=\mathrm{min}\left \{\mathcal{O}_{\vec{v}}(\epsilon)|~\vec{u}_j\in \mathcal{O}_{\vec{v}}(\epsilon),~\epsilon>0\right \}\label{choice 1}
\end{align}
for each $\vec{u}_j\in \cup_{i=1}^{\infty}\{\vec{u}_i\}$. By reason of our choice, there exists a magnet $\vec{u}_t\in \mathcal{O}_{\vec{v}}(\delta)$ for $\vec{u}_t\in \cup_{i=1}^{\infty}\{\vec{u}_i\}$ since magnets are dense on the boundary $\mathcal{C}_{\mathcal{B}}$ by definition \ref{magnetic boundary}. Next, we apply the magnetization $\Lambda:\mathcal{C}_{\mathrm{Int}}\longrightarrow \mathbb{V}^n$ and obtain the exit path 
\begin{align}
\Lambda_{\cup_{i=1}^{\infty}\{\vec{u}_i\}}(\vec{v})=\Lambda_{\vec{u}_t}(\vec{v})=\vec{u}_t-\vec{v}\nonumber
\end{align}
with $\vec{v}\cdot \vec{u}_t=0$ using Theorem \ref{neibourhood magnetization} with the least measure of magnetization
\begin{align}
||\vec{v}-\vec{u}_t||=\mathrm{min}\left \{||\vec{v}-\vec{u}_s||:\vec{u}_s\in \cup_{i=1}^{\infty}\{\vec{u}_i\}\right \}.\nonumber
\end{align}
Thus, the hiker exits the forest without being given information on his orientation within the forest with the shortest path $\vec{u}_t-\vec{v}$.
\bigskip

It must be said that the solution as espoused in the sketch may not be viewed as a complete solution to the Bellman lost in the forest problem, because in practice the equivalence of the minimal distance of the hiker $\vec{v}$ to the boundary of the simple close curve (forest) $\mathcal{C}_{\mathcal{B}}$ with the notion of magnetization $\Lambda:\mathcal{C}_{\mathrm{Int}}\longrightarrow \mathbb{V}_n$ under the additional orthogonality condition of the hiker and some point on the boundary of the forest $\vec{v}\cdot \vec{u}_t=0$ may not necessarily hold. Nevertheless, we believe that it is still possible that the problem could be studied under the equivalence
\begin{align}
 \Lambda_{\vec{u}_1,\vec{u}_2,\ldots,\vec{u}_k}(\vec{v})=\Lambda_{\vec{u}_j}(\vec{v})=\vec{u}_j-\vec{v}\nonumber
\end{align}
if and only if
$$
||\vec{v}-\vec{u}_j||=\mathrm{min}\left \{||\vec{u}_s-\vec{v}||\right \}_{s=1}^{k}
$$ 
without the extra regime $\vec{v}\cdot \vec{u}_j=0$ for any hiker in the forest.
\bigskip

\begin{tikzpicture}[>=Latex, line cap=round, line join=round, scale=1]

\draw[thick] (0,0) ellipse [x radius=4cm, y radius=3cm];

\node[font=\small] at (0,3.45) {forest (simple closed curve)};

\foreach \a in {0,10,...,350}{
    \fill ({4*cos(\a)},{3*sin(\a)}) circle (1.5pt);
}

\coordinate (M) at (4,0);
\fill (M) circle (2.6pt);
\draw[very thick] (M) circle (4.5pt);
\node[above right=2pt of M, font=\small] {nearest / strongest magnet};

\coordinate (H) at (1.6,0);
\fill (H) circle (2.6pt);
\node[below left=2pt of H, font=\small] {hiker};

\draw[very thick,->] (H) -- (M);
\node[font=\small] at (2.7,0.35) {shortest path};

\node[align=center, font=\small] at (0,-3.7)
{Dense magnets are mounted on the boundary.\\
The hiker is pulled to the nearest one and exits through the shortest path.};

\end{tikzpicture}

\bibliographystyle{amsplain}

\begin{thebibliography}{10}

\bibitem {bellman1956minimization} R. Bellman, \textit{Minimization problem} Bull. Amer. Math. Soc, vol. 62:3, 1956, pp. 270.\\

\bibitem {finch2004lost} S.R. Finch and J.E. Wetzel, \textit{Lost in a forest} The American Mathematical Monthly, vol. 111:8, Taylor \& Francis, 2004, pp. 645--654.\\

\bibitem {ward2008exploring} J. Ward, \textit{Exploring the Bellman Forest Problem}, 2008.

\end{thebibliography}

\end{document}